\documentclass[a4paper,intlimits]{amsart}

\usepackage[all]{paper_diening}

\usepackage{amsmath,amsthm,amssymb}

\usepackage{epsfig}
\usepackage[utf8]{inputenc}
\usepackage[T1]{fontenc}
\usepackage{dsfont}
\usepackage{faktor}

\usepackage{mathrsfs}

\usepackage[symbol]{footmisc}

\numberwithin{equation}{section}

\newcommand{\Div}{\divergence}
\newcommand{\R}{\mathbb R}

\newcommand{\dd}{\,\mathrm{d}}
\newcommand{\ds}{\dd s}
\newcommand{\dt}{\dd t}
\newcommand{\dx}{\dd x}

\providecommand{\seminormtmp}[2]{{#1[{#2}#1]}}
\providecommand{\seminorm}[1]{\seminormtmp{}{#1}}

\newcommand{\com}[1]{\color{black}#1\color{black}}

\begin{document}

\title{On temporal regularity for strong solutions to stochastic $p$-laplace systems}

\author[J. Wichmann]{J\"{o}rn Wichmann}
\address{Department of Mathematics, University of Bielefeld,
  \\
  Postfach 10 01 31, 33501 Bielefeld, Germany}
\email{jwichmann@math.uni-bielefeld.de}
\thanks{jwichmann@math.uni-bielefeld.de\\
The research was funded by the Deutsche Forschungsgemeinschaft (DFG, German Research Foundation) – SFB 1283/2 2021 – 317210226.
}
%
%

\begin{abstract}
  { In this article we investigate the temporal regularity of strong solutions to the stochastic $p$-\com{L}aplace system in the degenerate setting, $p \in [2,\infty)$, driven by a multiplicative nonlinear stochastic forcing. We establish $1/2$ time differentiability in an expontential Besov-Orlicz space for the solution process $u$. Furthermore, we prove $1/2$ time differentiability of the nonlinear gradient $\abs{\nabla u}^\frac{p-2}{2} \nabla u$ in a Nikolskii space.
}
\vspace{2mm}

\end{abstract}

\maketitle

\section{Introduction}
Let $\mathcal{O} \subset \R^n$ be a bounded domain, $n$, $N \in \mathbb{N}$ \com{and } $T > 0$ be finite. We are interested in the time and spatial regularity of the solution process $u$ to the stochastic $p$-Laplace system. Given an initial datum $u_0$ and a stochastic forcing term $(G,W)$ (for the precise assumptions see Assumption \ref{ass:Noise}), $u$ is determined by the relations 
\begin{alignat}{2} \label{intro:p-Laplace-stoch}
\begin{aligned}
\dd u - \Div S(\nabla u) \dt  &= G(u) \dd W \quad &&\text{ in } (0,T) \times \mathcal{O}, \\
u &= 0 \quad &&\text{ on } (0,T) \times \partial \mathcal{O}, \\
u(0) &= u_0 &&\text{ on } \mathcal{O},
\end{aligned}
\end{alignat}
where $S(\xi):= \left( \kappa + \abs{\xi} \right)^{p-2} \xi$\com{, $\xi \in \R^{n \times N}$}, $p \in [2,\infty)$ and $\kappa \geq 0$. 

\com{
In applications, the time and space regularity of solutions to \eqref{intro:p-Laplace-stoch} are of great importance. When designing numerical algorithms, the regularity determines the rate of convergence of the scheme. In \cite{MR4286257} we  construct an algorithm for the determinstic $p$-Laplace system, that is able to approximate rough solutions. There one sees the delicate interplay between regularity of the solution and the rate of convergence of the algorithm. }

The existence of analytically weak solutions to \eqref{intro:p-Laplace-stoch} in the space 
\begin{align*}
L^2(\Omega; C([0,T]; L^2(\mathcal{O}))) \cap L^p(\Omega;L^p((0,T);W^{1,p}_0(\mathcal{O})))
\end{align*}
can be established by standard monotonicity arguments \cite{LiRo}. In the deterministic setting it is natural to obtain regularity estimates for strong solutions in the spaces
\begin{subequations}
\label{intro:instat}
\begin{align}
V(\nabla u) &\in L^2(0,T;W^{1,2}(\mathcal{O})) \cap W^{1,2}(0,T;L^2(\mathcal{O})), \label{intro:V-instat} \\
u &\in  L^\infty(0,T;W^{1,2}(\mathcal{O})) \cap C^{0,1}([0,T],L^2(\mathcal{O})), \label{intro:u-instat}
\end{align}
\end{subequations}
where $V(\xi):= \left(\kappa + \abs{\xi} \right)^{\frac{p-2}{2}} \xi$. For this we test the system formally with $-\Delta u$ and $\partial_t^2 u$. It can be made rigorous by a substitution of differentials by difference quotients, cf. \cite{BreMen18pre}. In the stochastic case, it is still possible to prove local spatial regularity, i.e.
\begin{subequations} \label{intro:loc}
\begin{align}
u &\in L^2(\Omega; L^\infty(0,T;W^{1,2}_{\text{loc}}(\mathcal{O})), \\
V(\nabla u) &\in L^2(\Omega \times (0,T); W^{1,2}_{\text{loc}}(\mathcal{O})),
\end{align}
\end{subequations}
using difference quotients, cf. \cite{Breit2015Regularity}. However, it is not clear whether the estimat\com{e } holds up to the boundary under appropiate assumptions on the geometry of $\mathcal{O}$. 

Another method to construct strong solutions is presented in \cite{Gess2012Strong}. There the author proves global spatial regularity using a Galerkin ansatz and a special projection operator associated to the energy of the system. He obtains uniform estimates in the space
\begin{subequations} \label{intro:div}
\begin{align}
u &\in L^\infty(0,T; L^p(\Omega;  W^{1,p}(\mathcal{O}))), \\
\Div S(\nabla u) &\in L^2((0,T)\times \Omega;L^2(\mathcal{O}))).
\end{align}
\end{subequations}
The substantial difference between the estimates \eqref{intro:loc} and \eqref{intro:div} is, that the former corresponds to a formal testing of the equation with $-\Delta u$ whereas the latter corresponds to testing with $-\Div S(\nabla u)$. The second approach fits naturally to the gradient flow structure of the $p$-Laplace system \eqref{intro:p-Laplace-stoch}.

For sufficient\com{ly } regular or convex domains it is possible to extend the regularity estimate on $\Div S(\nabla u)$ to $\nabla S(\nabla u)$ as presented in \cite{Balci2021}, \cite{MR4030249} and \cite{CiaMaz20}. This allows to establish, at least in the non-degenerate setting $\kappa > 0$, global gradient regularity for the nonlinear expression $V(\nabla u)$, cf. Corollary \ref{cor:nabla-V-reg}.

In contrast to the spatial regularity, where deterministic tools can be applied, we need different techniques in order to investigate the time regularity of strong solutions. In the stochastic setting, the time regularity of the deterministic system \eqref{intro:instat} is unreachable, due to the irregularity of the cylindrical Wiener process $W$. Even in the linear case, $p=2$, the time regularity is substantially lower compared to the deterministic system. 
The maximal regularity for the stochastic heat equation with an additive forcing has been established in \cite{MR2952092} (see also \cite{MR2241096}). They prove
\begin{align}
\mathbb{E} \left[ \norm{u}_{H^{\alpha,q}(0,T;L^2(\mathcal{O})}^q \right] \lesssim \mathbb{E} \left[ \norm{G}_{L^q(0,T;L^2(\mathcal{O})}^q \right],
\end{align}
where $\alpha \in [0,1/2)$, $q \in (2,\infty)$ and $H$ denotes a Bessel potential space. The authors use the concept of mild solutions and estimate a stochastic convolution operator. 

In the pure nonlinear setting, the approach based on mild solutions does not fit anymore. In \cite{BH16_Monotone} the authors conjectured, that it is possible to establish 
\begin{subequations} \label{intro:Walpha-2}
\begin{align}
u &\in L^2(\Omega;C^{\alpha}(0,T;L^2(\mathcal{O}))), \\
V(\nabla u) &\in L^2(\Omega; W^{\alpha,2}(0,T;L^2(\mathcal{O}))),
\end{align}  
\end{subequations}
$\alpha  \in [0,1/2)$. In this paper, we do not only verify \eqref{intro:Walpha-2}, but \com{in addition } improve the result to a stronger scale of spaces (cf. Theorem \ref{thm:Besov-reg} and Theorem \ref{thm:time-reg-sol})
\begin{subequations}
\label{intro:instat-stoch}
\begin{align}
u &\in  L^2(\Omega; B^{1/2}_{\Phi_2,\infty}(0,T;L^2(\mathcal{O}))), \label{intro:u-instat-stoch}\\
V(\nabla u) &\in L^2(\Omega;B_{2,\infty}^{1/2}(0,T;L^2(\mathcal{O}))), \label{intro:V-instat-stoch}
\end{align}
\end{subequations}
where $B$ denotes a Besov space (for more details see Section \ref{sec:Function spaces}) and $\Phi_2(t) = e^{t^2} - 1$. \com{The time regularity \eqref{intro:u-instat-stoch} is optimal, in the sense that it exactly matches the time regularity of the underlying Wiener process $W$ as presented in \cite{VerHyt08}. } 

A key ingredient is the stability of the stochastic integral generated by a cylindrical Wiener process $W$ in type $2$ Banach spaces as explained in \cite{MR4116708}. The gradient regularity \eqref{intro:V-instat-stoch} heavily relies on the $V$-coercivity,
\begin{align}
\abs{V(\xi_1) - V(\xi_2)}^2 \eqsim \left( S(\xi_1) - S(\xi_2) \right) : (\xi_1 - \xi_2),
\end{align}
for all $\xi_1,\xi_2 \in \R^{n \times N}$ and the boundary condition of the nonlinear operator $G$. Furthermore, we show the improved time regularity
\begin{align}
V(\nabla u) \in L^2(\Omega; B_{q,\infty}^{1/2}(0,T;L^2(\mathcal{O})))
\end{align}
for $q > 2$, if the diffusion operator has improved time integrability, i.e. $\Div S(\nabla u) \in L^2(\Omega; L^q(0,T; L^2(\mathcal{O})))$ (see Theorem \ref{thm:higher-reg}).

In the past many authors have studied variants of \eqref{intro:p-Laplace-stoch} under different perspectives. The literature on the numerical analysis of the deterministic system is rich  \cite{BaLi1,Wei1992,BaLi2,EbLi,DER,BarDieNoc20}.

 The corresponding numerical analysis of the stochastic system is not that well developed. First results of fully implementable approximations have been discussed in \cite{MR2139212,MR2465711}. Just recently, an algorithm for the stochastic system has been proposed in \cite{MR4298537}. They infer convergence of the algorithm under the assumption \eqref{intro:Walpha-2} and the global spatial regularity assumption
\begin{subequations}
\begin{align}
u &\in L^2(\Omega; L^\infty(0,T;W^{1,2}(\mathcal{O}))), \\
V(\nabla u) & \in L^2(\Omega; L^2(0,T;W^{1,2}(\mathcal{O}))).
\end{align}
\end{subequations}

Well-posedness \com{of \eqref{intro:p-Laplace-stoch} } with merely $L^1$-initial data has been addressed in \cite{MR4225916}. The more general system, where $p$ is allowed to depend on $(\omega,t,x)$ respectively on $(t,x)$, is considered in \cite{MR3585706}, \cite{MR3535765} respectively in \cite{MR3327516}. The singular case $p \in [1,2)$ has been analyzed in \cite{MR2561270}, \cite{MR3580815} and \cite{MR4261330}.

The paper is organized as follows. Section~\ref{sec:math_setup} introduces the mathematical setup and preliminary results on stochastic integrals and strong solutions. 
 Time regularity for strong solutions is addressed in Section~\ref{sec:Time_reg_Main}. It starts with stability results for stochastic integrals in Section~\ref{sec:Stability_of_stochastic_integrals}. In Section~\ref{sec:Besov_regularity} \com{ we prove Theorem~\ref{thm:Besov-reg} on } Besov regularity of strong solution. Section~\ref{sec:Nikolskii_regularity} is about regularity of the nonlinear gradient $V(\nabla u)$ \com{as presented in Theorem~\ref{thm:time-reg-sol}}. Lastly, \com{the improved estimates on the nonlinear gradient are discussed in Theorem~\ref{thm:higher-reg} in } Section~\ref{sec:Higher_order_Nikolskii}.

\section{Mathematical setup} \label{sec:math_setup}
In this section we introduce the setup for the system \eqref{intro:p-Laplace-stoch}. Let $\mathcal{O} \subset \R^n$, $n \geq 2$, be a bounded domain (further assumptions on $\mathcal{O}$ will be needed for the spatial regularity of solutions). For some given $T>0$ we denote by $I := \com{[}0,T\com{]}$ the time interval and write $\mathcal{O}_T := I \times \mathcal{O}$ for the time space cylinder. Moreover let $\left(\Omega,\mathcal{F}, (\mathcal{F}_t)_{t\in I}, \mathbb{P} \right)$ denote a stochastic basis, i.e. a probability space with a complete and right continuous filtration $(\mathcal{F}_t)_{t\in I}$. We write $f \lesssim g$ for two non-negative quantities $f$ and $g$ if $f$ is bounded by $g$ up to a multiplicative constant. Accordingly we define $\gtrsim$ and $\eqsim$. We denote by $c$ a generic constant which can change its value from line to line.
\subsection{Function spaces} \label{sec:Function spaces}
As usual\com{, } $L^q(\mathcal{O})$ denotes the Lebesgue space and $W^{1,q}(\mathcal{O})$ the Sobolev space, where $1\leq q \leq \infty$. We denote by $W^{1,q}_0(\mathcal{O})$ the Sobolev spaces with zero boundary values. It is the closure of $C^\infty_0(\mathcal{O})$ (smooth functions with compact support) in the $W^{1,q}(\mathcal{O})$-norm. We denote by $W^{-1,q'}(\mathcal{O})$ the dual of $W^{1,q}_0(\mathcal{O})$. We do not distinguish in the notation between vector- and matrix-valued functions.

For a Banach space $\left(X, \norm{\cdot}_X \right)$ let $L^q(I;X)$ be the Bochner space of Bochner-measurable functions $u: I \to X$ satisfying $t \mapsto \norm{u(t)}_X \in L^q(I)$. Moreover, $C(\overline{I};X)$ is the space of continuous functions with respect to the norm-topology. We also use $C^{\alpha}(\overline{I};X)$ for the space of H\"older continuous functions. Given an Orlicz-function $\Phi: [0,\infty] \to [0,\infty]$, i.e. a convex function satisfying $ \lim_{t \to 0} \Phi(t)/t = 0$ and $\lim_{t \to \infty} \Phi(t)/t = \infty$ we define the Luxemburg-norm 
\begin{align*}
\norm{u}_{L^\Phi(I;X)} := \inf \left\{ \lambda > 0 : \int_I \Phi \left( \frac{\norm{u}_X}{\lambda} \right) \ds \leq 1 \right\}.
\end{align*}
The Orlicz space $L^\Phi(I;X)$ is the space of all Bochner-measurable functions with finite Luxemburg-norm. \com{For more details on Orlicz-spaces we refer to \cite{DiHaHaRu}. Given }$h \in I$ and $u :I \to X$ we define the difference operator $\tau_h: \set{u: I \to X} \to \set{u: I\cap I - \set{h} \to X} $ via $\tau_h(u) (s) := u(s+h) - u(s)$. The Besov-Orlicz space $B^\alpha_{\Phi,r}(I;X)$ with differentiability $\alpha \in (0,1)$, integrability $\Phi$ and fine index $r \in (1,\infty]$ is defined as the space of Bochner-measurable functions with finite Besov-Orlicz norm $\norm{\cdot}_{B^\alpha_{\Phi,r}(I;X)}$, where
\begin{align*}
\norm{u}_{B^\alpha_{\Phi,r}(I;X)} &:= \norm{u}_{L^{\Phi}(I ;X)} + \seminorm{u}_{B^\alpha_{\Phi,r}(I;X)}, \\
\seminorm{u}_{B^\alpha_{\Phi,r}(I;X)} &:= \left( \int_{I} h^{-r\alpha} \norm{\tau_h u}_{L^\Phi(I\cap I - \set{h};X)}^r \dd h \right)^\frac{1}{r}.
\end{align*} 
In the case $r = \infty$ the integral in $h$ is replaced by an essential supremum and the space is commonly called Nikolskii-Orlicz space. When $\Phi(t) = t^p$ for some $p \in (1,\infty)$ we call the space $B^\alpha_{\Phi,r}(I;X) =B^\alpha_{p,r}(I;X)$ Besov space. Similarly, given a Banach space $\left(Y, \norm{\cdot}_Y \right)$, we define $L^q(\Omega;Y)$ as the Bochner space of Bochner-measurable functions $u: \Omega \to Y$ satisfying $\omega \mapsto \norm{u(\omega)}_Y \in L^q(\Omega)$. The space $L^q_{\mathcal{F}}(\Omega \times I;X)$ denotes the subspace of $X$-valued progressively measurable processes. Let $\left(U,\norm{\cdot}_U \right)$ be a separable Hilbert space. $L_2(U;L^2(\mathcal{O}))$ denotes the space of Hilbert-Schmidt operators from $U$ to $L^2(\mathcal{O})$ with the norm $\norm{z}_{L_2(U;L^2_x)}^2:= \sum_{j\in \mathbb{N}} \norm{z(u_j)}_{L^2(\mathcal{O})}^2 $ where $\set{u_j}_{j \in \mathbb{N}}$ is some orthonormal basis of $U$. We abbreviate the notation $L^q_\omega L^q_t L^q_x := L^q(\Omega;L^q(I;L^q(\mathcal{O}))) $ and $L^{q-} := \bigcap_{r< q} L^r$.

\subsection{Stochastic integrals} \label{sec:stoch_integral}
In order to construct the stochastic forcing term, we impose the following conditions:
\begin{assumption}\label{ass:Noise}
\begin{enumerate}
\item We assume that $W$ is an $U$-valued cylindrical Wiener process with respect to $(\mathcal{F}_t)_{t\in I}$. Formally $W$ can be represented as
\begin{align} \label{rep:W}
W = \sum_{j \in \mathbb{N}} u_j \beta^j,
\end{align}
where $\set{\beta^j}_{j\in \mathbb{N}}$ are independent $1$-dimensional standard Brownian motions.
\item \label{ass:cond2} Let $v \in L^2_{\mathcal{F}}(\Omega \times I; L^2_x)$. We assume that $G(v)(\cdot): U \to L_{\mathcal{F}}^2(\Omega \times I; L^2_x)$ is given by 
\begin{align*}
u \mapsto G(v)(u) := \sum_{j \in \mathbb{N}} g_j(\cdot,v)( u_j, u)_U ,
\end{align*}
where $\set{g_j}_{j\in \mathbb{N}} \in C^1(\mathcal{O} \times \R^N; \R^N)$ with
\begin{enumerate}[label=(\roman*)]
\item (sublinear growth) for all $x \in \mathcal{O}$ and $\xi \in \R^N$ it holds
\begin{align}\label{ass:growth}
\sum_{j\in \mathbb{N}} \abs{g_j(x,\xi)}^2 + \abs{\nabla_x g_j(x,\xi)}^2 \leq c_{\text{growth}}(1+\abs{\xi}^2),
\end{align} 
\item (Lipschitz \com{continuity}) for all $x \in \mathcal{O}$ and $\xi \in \R^N$ it holds
\begin{align} \label{ass:Lipschitz}
\sum_{j\in \mathbb{N}} \abs{\nabla_{\xi} g_j(x,\xi)}^2 \leq c_{\text{lip}}.
\end{align} 
\item (boundary data) for all $x\in \partial \mathcal{O}$, $\xi \in \R^N$ and $j \in \mathbb{N}$ it holds $g_j(x,\xi) = 0$. \label{ass:boundary-noise}
\end{enumerate}
\end{enumerate}
\end{assumption}
Commonly the stochastic integral is constructed under weaker assumptions on the coefficient $G$, e.g. $g_j \in C(\mathcal{O} \times \R^N; \R^N)$ with sublinear growth and a direct Lipschitz assumption. Then the operator  
\begin{align*}
G : L^2_{\mathcal{F}}(\Omega \times I; L^2_x) \to L^2_{\mathcal{F}}(\Omega \times I; L_2(U;L^2_x))
\end{align*}
is bounded and continuous. However we are interested in strong solutions and thus also need regularity \com{of } the gradient. For this we use the more involved stochastic integration theory in \com{type 2 } Banach spaces as done by Brze\'{z}niak \cite{MR1313905} and Van Neerven and Weis in \cite{MR2109586}. Let $\left(E,\norm{\cdot}_E \right)$ be a Banach space, $\set{\gamma_j}_{j\in \mathbb{N}} \sim \mathcal{N}(0,1)$ \com{independent and identically distributed } random variables on some probability space $\left(\Omega_\gamma, \mathcal{F}_\gamma,\mathbb{P}_\gamma \right)$ and $F: U \to E$ a linear operator. We define the norm
\begin{align}
\norm{F}_{\gamma(U;E)}^2:= \mathbb{E}_{\gamma} \left[ \norm{\sum_{j \in \mathbb{N}} \gamma_j F(u_j) }_E^2 \right]
\end{align}
and $\gamma(U;E) := \set{F: U \to E| F \text{ linear,} \, \norm{F}_{\gamma(U;E)} < \infty}$ as the space of $\gamma$-radonifying operators from $U$ to $E$. For a survey on $\gamma$-radonifying operators see \cite{MR2655391}. In our application we have $E = W^{1,p}_{\com{0},x}$.
\begin{lemma} \label{lem:G-gamma}
Let Assumption \ref{ass:Noise} be satisfied. Then
\begin{align*}
G: L^p_\mathcal{F}(\Omega \times I; W^{1,p}_{\com{0},x}) \to L^p_\mathcal{F}(\Omega \times I; \gamma(U;W^{1,p}_{0,x}))
\end{align*}
is bounded.
\end{lemma} 
\begin{proof}
Let $v \in  W^{1,p}_{\com{0},x}$ and $J\in \mathbb{N}$. Define the truncated operator
\begin{align*}
u \mapsto G^J(v)(u):= \sum_{j=1}^J g_j(\cdot,v) (u,u_j)_U.
\end{align*}
Due to the Kahane-Khintchine inequalities, Fubini's Theorem and the assumptions \eqref{ass:growth} and \eqref{ass:Lipschitz},
\begin{align*}
&\norm{G^J(v)}_{\gamma(U;W^{1,p}_x)} = \left( \mathbb{E}_{\gamma} \left[\norm{\sum_{j =1}^J \gamma_j g_j(\cdot, v)}_{W^{1,p}_x}^2 \right]\right)^\frac{1}{2} \\
&\eqsim \left( \mathbb{E}_{\gamma} \left[\norm{\sum_{j =1}^J \gamma_j g_j(\cdot, v)}_{W^{1,p}_x}^p \right] \right)^\frac{1}{p}  \\
&= \left( \int_{\mathcal{O}}\mathbb{E}_{\gamma} \left[ \abs{\sum_{j =1}^J \gamma_j g_j(\cdot, v)}^p + \abs{\sum_{j =1}^J \gamma_j (\nabla_x g_j(\cdot, v) + \nabla_\xi g_j(x,v) \nabla v)}^p\right] \dx  \right)^\frac{1}{p} \\
&\eqsim \left(  \int_{\mathcal{O}} \left(\mathbb{E}_{\gamma} \left[ \abs{\sum_{j=1}^J \gamma_j g_j(\cdot, v)}^2 \right] \right)^\frac{p}{2} \hspace{-6pt}+\left( \mathbb{E}_{\gamma} \left[ \abs{\sum_{j=1}^J \gamma_j (\nabla_x g_j(\cdot, v) + \nabla_\xi g_j(x,v) \nabla v)}^2\right] \right)^\frac{p}{2} \hspace{-10pt}\dx  \right)^\frac{1}{p} \\
&=\left(  \int_{\mathcal{O}} \left( \sum_{j=1}^J \abs{ g_j(\cdot, v)}^2 \right)^\frac{p}{2} +\left(  \sum_{j =1}^J  \abs{\nabla_x g_j(\cdot, v) + \nabla_\xi g_j(x,v) \nabla v}^2 \right)^\frac{p}{2} \dx  \right)^\frac{1}{p}\\
&\lesssim \norm{v}_{W^{1,p}_x} + 1,
\end{align*}
\com{where the constant is independent of $J$. } \com{Due to lower semicontinuity of the norm, we can pass to the limit}
\begin{align*}
\norm{G(v)}_{\gamma(U;W^{1,p}_x)}\lesssim \norm{v}_{W^{1,p}_x} + 1.
\end{align*}
For general $v \in L^p_\mathcal{F}(\Omega \times I; W^{1,p}_x)$ we apply the above result pointwise, i.e.
\begin{align*}
\norm{G(v)}_{L^p(\Omega \times I;\gamma(U;W^{1,p}_x))}^p &=  \mathbb{E} \left[ \int_0^T \norm{G(v)}_{\gamma(U;W^{1,p}_x)}^p \ds \right]\lesssim \mathbb{E} \left[ \int_0^T \norm{v}_{W^{1,p}_x}^p \ds \right] +1.
\end{align*}
The zero boundary data for $G(v)$ follows by \ref{ass:boundary-noise} in \com{A}ssumption \ref{ass:Noise}.
\end{proof}

\com{In the construction of the stochastic integral, the geometry of the underlying Banach space plays an important role. } 
\begin{definition} \label{def:type-q}
Let $\left(E,\norm{\cdot}_E \right)$ be a Banach space and $\set{\gamma_j}_{j\in \mathbb{N}} \sim \mathcal{N}(0,1)$ \com{independent and identically distributed. }
\begin{enumerate}
\item $E$ is of type $q \in [1,2]$ if there exists a constant $C >0 $ such that for all finite sequences $\set{e_j}_{j\in \mathbb{N}} \in E$
\begin{align} \label{eq:type-q1}
\left( \mathbb{E}_\gamma \left[ \norm{\sum_{j \in \mathbb{N}} \gamma_j e_j}_E^2 \right] \right)^\frac{1}{2} \leq C\left( \sum_{j \in \mathbb{N}} \norm{e_j}_E^q \right)^\frac{1}{q}.
\end{align}
\item $E$ is of cotype $q \in [2,\infty ]$ if there exists a constant $C >0 $ such that for all finite sequences $\set{e_j}_{j\in \mathbb{N}} \in E$
\begin{align} \label{eq:type-q2}
\left( \sum_{j \in \mathbb{N}} \norm{e_j}_E^{\com{q}} \right)^\frac{1}{\com{q}} \leq C\left( \mathbb{E}_\gamma \left[ \norm{\sum_{j \in \mathbb{N}} \gamma_j e_j}_E^2 \right] \right)^\frac{1}{2}.
\end{align}
\end{enumerate}
\end{definition}
Hilbert spaces are of type $2$ and of cotype $2$. Even the converse is true, if a Banach space is type $2$ and cotype $2$, then it is isomorphic to a Hilbert space, cf. \cite{MR341039}. On the classical Lebesgue scale the following result is valid.
\begin{proposition}[\cite{MR3617459} Proposition 10.36] \label{prop:type2}
If $q \in [1,2]$, every $L^q$-space is of type $q$ and of cotype $2$. If $q \in [2,\infty)$, every $L^q$-space is of type $2$ and of cotype $q$.
\end{proposition}

Now we can construct the stochastic integral.
\begin{proposition}\label{prop:ConstructionStochIntegral}
Let Assumption \ref{ass:Noise} be true. Then the operator $\mathcal{I}$ defined through
\begin{align}\label{def:stoch_integral}
\mathcal{I}(G(v)) := \int_0^{\bullet} G(v)(\dd W_s) := \sum_{j \in \mathbb{N}} \int_0^{\bullet} g_j(\cdot,v) \dd \beta^j_s 
\end{align} 
defines a bounded linear operator from $L^2_{\mathcal{F}}(\Omega \times I;L_2(U; L^2_x))$ to $ L^2_\omega C_t L^2_x $. Moreover,
\begin{itemize}
\item $\mathcal{I}(G(v))$ is an $L^2_x$-valued martingale with respect to $\left(\mathcal{F}_t\right)_{t \in I}$, 
\item (It\^o isometry) for all $t \in I$ it holds
\begin{align*}
\mathbb{E} \left[ \norm{\mathcal{I}(G(v)) (t)}_{L^2_x}^2 \right] =\mathbb{E} \left[\int_0^t \norm{G(v)}^2_{L_2(\com{U};L^2_x)} \ds \right] .
\end{align*}
\end{itemize}
Let $p \geq 2$. Then we have additionally the gradient estimate
\begin{align}\label{eq:Gradient_estimate_stoch_integral}
 \left( \mathbb{E} \left[\sup_{t \in I} \norm{\mathcal{I}(G(v))(t)}_{W^{1,p}_x}^p \right] \right)^\frac{1}{p}\lesssim \com{\sqrt{p} \left( \norm{v}_{L^p_\omega L^2_t W^{1,p}_x} + 1\right)}.
\end{align}
\end{proposition}
\begin{proof}
The first part of Proposition \ref{prop:ConstructionStochIntegral} is standard and we skip its proof. 

\com{In order to prove  \eqref{eq:Gradient_estimate_stoch_integral} we employ the stability of the stochastic integral in type 2 Banach spaces as presented in \cite{MR4116708}. Let $v \in L^p(\Omega \times I; W^{1,p}_{0,x})$. Lemma \ref{lem:G-gamma} ensures that $G(v)$ is stochastically integrable on $W^{1,p}_{0,x}$. Thus, 
\begin{align*}
\left( \mathbb{E} \left[ \sup_{t \in I} \norm{\mathcal{I}(G(v))(t)}_{W^{1,p}_{x}}^p \right] \right)^\frac{1}{p} &\lesssim \sqrt{p} \norm{G(v)}_{L^p_\omega L^2_t \gamma(U;W^{1,p}_x)} \\
&\lesssim \sqrt{p} \left( \norm{v}_{L^p_\omega L^2_t W^{1,p}_x} + 1\right).
\end{align*} 
}

\end{proof}

\begin{remark}
In the regime $p \in (1,2)$ the inequality \eqref{eq:Gradient_estimate_stoch_integral} fails. This is directly linked to the fact that $L^r$, $r \in (1,2)$ is not of type $2$. \com{However, if one uses the stochastic integration theory in UMD (unconditional martingale differences) Banach spaces as done in \cite{MR2330977}, it is still possible to estimate }
\begin{align*}
&\mathbb{E} \left[ \sup_{t \in I} \norm{\mathcal{I}(G(v))(t)}_{W^{1,p}_x}^p \right] \lesssim \mathbb{E} \left[  \int_\mathcal{O} \left(\int_0^T 1 + \abs{v}^2 + \abs{\nabla v}^2 \dd s \right)^\frac{p}{2}\dx \right],
\end{align*}
i.e. $\mathcal{I} \circ G: L^p_\omega W^{1,p}_x L^2_t \rightarrow L^p_\omega C_t W^{1,p}_x$ is bounded. 
\end{remark}

\subsection{Perturbed gradient flow} \label{sec:gradient_flow}
Let $\kappa \geq 0$ and $p \in \com{[2},\infty)$. For $\xi \in \R^{n\times N}$ we define
\begin{align} \label{eq:S}
S(\xi) := \varphi'(\abs{\xi}) \frac{\xi}{\abs{\xi}} = \left(\kappa + \abs{\xi} \right)^{p-2} \xi
\end{align}
and
\begin{align}\label{eq:V}
V(\xi) := \sqrt{\varphi'(\abs{\xi})} \frac{\xi}{\abs{\xi}} = \left( \kappa + \abs{\xi} \right)^\frac{p-2}{2} \xi,
\end{align}
where $\varphi(t):= \int_0^t \left(\kappa + s \right)^{p-2} s \ds$. The nonlinear functions $S$ and $V$ are closely related. In particular the following Lemmata are of great imp\com{or}tance. \com{The proofs can be found in \cite{DieE08}. For more details we refer to \cite{DR,BelDieKre2012,DieForTomWan20}.}
\begin{lemma}[$V$-coercivity] \label{lem:V-coercive}
Let $\xi_1, \xi_2 \in \R^{n\times N}$. Then it holds
\begin{align} \label{eq:V-coercive}
\begin{aligned}
(S(\xi_1) - S(\xi_2)) : (\xi_1 - \xi_2) &\eqsim \abs{V(\xi_1) - V(\xi_2)}^2 \\
&\eqsim \left( \kappa + \abs{\xi_1} + \abs{\xi_1 - \xi_2} \right)^{p-2} \abs{\xi_1 - \xi_2}^2.
\end{aligned}
\end{align}
\end{lemma}
\begin{lemma}[generalized Young's inequality] \label{lem:gen-young}
Let $\xi_1, \xi_2, \xi_3 \in \R^{n\times N}$ and $\delta >0$. Then there exists $c_\delta \geq 1$ such that
\begin{align}\label{eq:gen-young}
\left( S(\xi_1) - S(\xi_2) \right): \left(\xi_2 - \xi_3 \right) \leq \delta \abs{V(\xi_1) - V(\xi_2)}^2 + c_\delta \abs{V(\xi_2) - V(\xi_3)}^2.
\end{align} 
\end{lemma}
\begin{lemma} \label{lem:1side-young}Let $\xi_1, \xi_2, \xi_3 \in \R^{n\times N}$ and $\delta >0$. Then there exists $c_\delta \geq 1$ such that
\begin{align} \label{eq:1side-young}
\left( S(\xi_1) - S(\xi_2) \right) : \xi_3 \leq \delta \abs{V(\xi_1) - V(\xi_2)}^2 \hspace{-2pt}+ c_\delta \left(\kappa + \abs{\xi_1} + \abs{\xi_1 - \xi_2}\right)^{p-2}\abs{\xi_3}^2.
\end{align}
\end{lemma}
\begin{remark}
Lemma \ref{lem:V-coercive} and \ref{lem:gen-young} are still valid if one replaces $\varphi$ in \eqref{eq:S} and \eqref{eq:V} by any uniformly convex $N$-function. 
\end{remark} 
Given some \com{$\mathcal{F}_0$-measurable } initial condition $u_0 : \Omega \times \mathcal{O} \to \R^N$ and a stochastic force $(G,W)$ in the sense of Assumption \ref{ass:Noise}\com{, } we are interested in the system
\begin{subequations}
\label{eq:p-Laplace}
\begin{alignat}{2}\label{eq:Gradient_flow}
\dd u - \Div S(\nabla u) \dt &= G(u) \dd W &&\text{ in } \Omega \times \mathcal{O}_T, 
\end{alignat}
with boundary and initial conditions given by 
\begin{alignat}{2}
u &= 0 &&\text{ on } \Omega \times I \times \partial \mathcal{O}, \\
u(0) &= u_0 &&\text{ on } \Omega \times  \mathcal{O}.
\end{alignat}
\end{subequations}
The system \eqref{eq:Gradient_flow} is a perturbed version of the gradient flow of the energy $\mathcal{J}: W^{1,p}_{0,x} \to [0,\infty)$ given by
\begin{align} \label{eq:energy}
\mathcal{J}(u):= \int_{\mathcal{O}} \varphi(\abs{\nabla u}) \dx.
\end{align}
\subsection{Weak and strong solutions} \label{sec:solution}
We fix the concept of solutions as follows.
\begin{definition} \label{def:sol}
Let \com{$u_0 \in L^2_\omega L^2_x$ be $\mathcal{F}_0$-measurable, $p \geq 2$ } and $(G,W)$ \com{be } given by Assumption \ref{ass:Noise}. An $(\mathcal{F}_t)$-adapted process $u \in L^2_x$ is called weak solution to \eqref{eq:p-Laplace} if 
\begin{enumerate}
\item $u \in \com{L^2_\omega} C_t L^2_x \cap \com{L^p_\omega} L^p_t W^{1,p}_{0,x}$,
\item for all $t \in I$, $\xi \in C^\infty_{0,x}$ \com{ and $\mathbb{P}-a.s.$ } it holds
\begin{align} \label{eq:p-Laplace_weak}
\int_{\mathcal{O}} (u(t) - u_0) \cdot \xi \dx + \int_0^t \int_{\mathcal{O}} S(\nabla u) :\nabla \xi \dx \ds = \int_{\mathcal{O}} \int_0^t G(u)\dd W_s \cdot \xi \dx.
\end{align}
\end{enumerate}
The process $u$ is called strong solution if it is a weak solution and additionally satisfies
\begin{enumerate}
\item $\Div S(\nabla u) \in \com{L^2_\omega}L^2_t L^2_x$ ,
\item for all $t \in I$ \com{ and $\mathbb{P}-a.s.$ } it holds
\begin{align}
u(t) - u_0 - \int_0^t \Div S(\nabla u) \ds = \int_0^t G(u) \dd W_s
\end{align}
as an equation in $L^2_x$.
\end{enumerate}
\end{definition}
The existence of strong solutions to gradient flow like equations has been established by Gess in \cite{Gess2012Strong}. In particular, it includes the case of the $p$-Laplace system for $p \geq 2$.
\begin{theorem}[\cite{Gess2012Strong} Theorem 4.12] \label{thm:strong-sol}
Assume $p\geq 2$ and $\mathcal{O}$ \com{to be } a bounded convex domain. Let $(G,W)$ be given by Assumption \ref{ass:Noise} and $u_0 \in L^p_\omega W^{1,p}_x$ be $\mathcal{F}_0$-measurable. Then there exists a unique strong solution $u$ to \eqref{eq:p-Laplace}. Moreover,
\begin{align}\label{eq:est-strong}
\mathbb{E}\left[\sup_{t\in I} \norm{u}_{W^{1,p}_x}^p  + \int_0^T \norm{\Div S(\nabla u)}_{L^2_x}^2 \dt \right] \lesssim \mathbb{E} \left[ \norm{u_0}_{W^{1,p}_x}^p \right] + 1.
\end{align}
\end{theorem}
\begin{proof}
We will only sketch the idea of the proof. Similar to the construction of weak solutions one first performs a Galerkin ansatz $V_J \subset W^{1,p}_{0,x}$. Solving the corresponding SDE provides an approximation $u_J$. Apart from the classical a priori estimates $u_J \in L^2_\omega L^\infty_t L^2_x \cap L^p_\omega L^p_t W^{1,p}_{0,x}$, \com{one } can also get a priori estimates on the gradient level $u_J \in L^p_\omega L^\infty_t W^{1,p}_x$ and $\Div S(\nabla u_J) \in L^2_\omega L^2_t L^2_x$ by an application of It\^o's formula \com{to } $u_J \mapsto \mathcal{J}(u_J)$. Finally, one passes to the limit $J \to \infty$ and identifies the nonlinear terms $S$ and $G$ using the monotonicity of $S$ and the Lipschitz assumption on $G$. Uniqueness \com{already holds for weak solutions}.
\end{proof}
In fact the result can be \com{strengthened } to not only give a bound on the divergence of $S(\nabla u)$ but on the full gradient. An optimal result in the vector valued setting has been obtained in \cite{Balci2021} Theorem 2.6. In contrast to the scalar case there exists a threshold $p_c = 2(2-\sqrt{2})$ such that the regularity transfer to the full gradient is only valid for $p> p_c$. We also want to mention an earlier result of Cianchi and Maz'ya \cite{MR4030249} where they obtained a similar but suboptimal result. 

The geometry of the domain $\mathcal{O}$ plays an important role. \com{Let $\mathcal{O}$ be a bounded Lipschitz domain } such that $\partial \mathcal{O} \in W^{2,1}$, i.e. $\mathcal{O}$ is locally the subgraph of a Lipschitz continuous function of $n-1$ variables, which is also twice weakly differentiable. Denote by $\mathcal{B}$ the second fundamental form on $\partial \mathcal{O}$, by $\abs{\com{\mathcal{B}}}$ its norm and define 
\begin{align}
\mathcal{K}_\mathcal{O}(r) := \sup_{E \subset \partial \mathcal{O} \cap B_r(x), x \in \partial \mathcal{O}} \frac{\int_E \abs{\mathcal{B}}\dd \mathcal{H}^{n-1}}{\text{cap}_{B_1(x)}(E)},
\end{align}
where $B_r(x)$ denotes the ball of radius $r$ around $x$, $\text{cap}_{B_1(x)}(E)$ is the capacity of the set $E$ relative to the ball $B_1(x)$ and $\mathcal{H}^{n-1}$ is the $n-1$ dimensional Hausdorff measure.

\begin{lemma} \label{lem:grad-S}
Let $p>p_c$. Assume that $\mathcal{O}$ is either 
\begin{enumerate}
\item bounded and convex,
\item or bounded, Lipschitz and $\partial \mathcal{O} \in W^{2,1}$ with $\lim_{r\to 0} \mathcal{K}_\mathcal{O}(r) \leq c$.
\end{enumerate}
Let $v \in L^p_\omega L^p_t W^{1,p}_{0,x}$ with $\Div S(\nabla v) \in L^2_\omega L^2_t L^2_x$. 

Then $\nabla S(\nabla v) \in L^2_\omega L^2_t L^2_x$ and
\begin{align*}
\mathbb{E} \left[ \norm{\nabla S(\nabla v)}_{L^2_t L^2_x}^2 \right] \eqsim \mathbb{E} \left[ \norm{\Div S(\nabla u)}_{L^2_t L^2_x}^2 \right].
\end{align*}
\end{lemma}
\com{
\begin{proof}
First observe that $v$ trivially solves for almost all $(\omega,t) $
\begin{alignat*}{2}
\Div S(\nabla v) &= \Div S(\nabla v) \quad &&\text{ in } \mathcal{O},\\
v &= 0 \quad &&\text{ on } \partial \mathcal{O}.
\end{alignat*}
The result now follows by Theorem~2.3 respectively Theorem~2.4 in \cite{MR4030249}.
\end{proof}
}

This allows to establish global gradient regularity for $V(\nabla u)$ at least in the non-degenerate setting $\kappa > 0$.
\begin{corollary}\label{cor:nabla-V-reg}
Let the assumptions of Theorem~\ref{thm:strong-sol} be satisfied. Denote by $u$ the unique strong solution to~\eqref{eq:p-Laplace}. Then
\begin{align}\label{eq:nabla-V-reg}
\mathbb{E}\left[ \norm{\nabla V(\nabla u)}_{L^2_t L^2_x}^2 \right] \lesssim \kappa^{2-p} \left(\mathbb{E} \left[ \norm{u_0}_{W^{1,p}_x}^p \right] + 1 \right).
\end{align}
\end{corollary}
\begin{proof}
We compute the derivatives
\begin{align*}
\nabla S(\nabla u) &= \left(\kappa + \abs{\nabla u} \right)^{p-2} \left( \nabla^2 u + (p-2) \frac{\nabla^2 u \nabla u \otimes \nabla u}{\left( \kappa + \abs{\nabla u}\right)\abs{\nabla u}} \right),\\
\nabla V(\nabla u) &= \left(\kappa + \abs{\nabla u} \right)^\frac{p-2}{2} \left( \nabla^2 u + \frac{p-2}{2} \frac{\nabla^2 u \nabla u \otimes \nabla u}{\left( \kappa + \abs{\nabla u}\right)\abs{\nabla u}} \right).
\end{align*}
In particular,
\begin{align*}
\abs{\nabla V(\nabla u)}^2 &\eqsim \left(\kappa + \abs{\nabla u} \right)^{p-2} \abs{\nabla^2 u}^2 \eqsim \nabla S(\nabla u) : \nabla^2 u,\\
\abs{\nabla S(\nabla u)}^2 &\eqsim \left(\kappa + \abs{\nabla u} \right)^{2(p-2)} \abs{\nabla^2 u}^2.
\end{align*}
H\"older's inequality now implies
\begin{align*}
\mathbb{E}\left[\norm{\nabla V(\nabla u)}_{L^2_t L^2_x}^2 \right] \lesssim \left( \mathbb{E}\left[\norm{\nabla S(\nabla u)}_{L^2_t L^2_x}^2 \right]\mathbb{E}\left[\norm{\nabla^2 u}_{L^2_t L^2_x}^2 \right] \right)^\frac{1}{2}
\end{align*}
Due to the non-degeneracy we have
\begin{align*}
\mathbb{E}\left[\norm{\nabla^2 u}_{L^2_t L^2_x}^2 \right] &= \mathbb{E}\left[ \int_0^T \int_\mathcal{O} \frac{\left( \kappa + \abs{\nabla u} \right)^{2(p-2)}}{\left( \kappa + \abs{\nabla u} \right)^{2(p-2)}} \abs{\nabla^2 u}^2 \dx \dt \right] \\
&\lesssim \kappa^{2(2-p)}\mathbb{E}\left[\norm{\nabla S(\nabla u)}_{L^2_t L^2_x}^2 \right].
\end{align*}
Since we assume that $p \geq 2 > p_c$ and $\mathcal{O}$ is bounded and convex, we can apply Lemma~\ref{lem:grad-S}. The estimate~\eqref{eq:nabla-V-reg} now follows by the a priori estimate of Theorem~\ref{thm:strong-sol}.
\end{proof}

\begin{remark}
Naturally one obtains regularity for $\nabla V(\nabla u)$ when testing \eqref{eq:p-Laplace} by the Laplacian, as presented by Breit in \cite{Breit2015Regularity}, \com{rather then the $p$-Laplacian. } In this way one can establish uniform bounds in $\kappa$ at least locally in space. For global estimates the geometry of $\mathcal{O}$ is the crucial ingredient.
\end{remark}

\section{Time regularity for strong solutions}\label{sec:Time_reg_Main}
The next paragraph is devoted to results on the time regularity of strong solutions to \eqref{eq:p-Laplace}. Especially important are the mapping properties of the stochastic integral operator in type $2$ Banach spaces. The regularity of the stochastic integral also improves when the regularity of the integrand improves. This is a key ingredient when deriving time regularity of solutions to SPDEs as presented by Ondrej\'at and Veraar in \cite{MR4116708}. 
\subsection{Stability of stochastic integrals} \label{sec:Stability_of_stochastic_integrals}
First we need a regularity statement for conditional expectations of stochastic integrals.
\begin{lemma}[\cite{MR4116708} Lemma 3.1] \label{lem:reg-cond-exp}
Let $\left( E, \norm{\cdot}_E \right)$ be a separable type $2$ Banach space. Let $T>0$, $q\in [1,\infty)$ and $r \in (2,\infty]$. Let $G \in L^q_\mathcal{F} L^r_t \gamma(U;E)$. Then there exists a constant $C>0$ such that for all $0 \leq s  \leq t \leq T$ and $\mathbb{P}$-a.s.
\begin{align} \label{eq:cond-exp}
\left( \mathbb{E}\left[\norm{\mathcal{I}(G)(t) - \mathcal{I}(G)(s)}_{E}^q | \mathcal{F}_s \right] \right)^\frac{1}{q} \leq C \sqrt{q} \norm{G}_{L^q_\omega L^r_t \gamma(U;E)} (t-s)^{\frac{1}{2} - \frac{1}{r}}.
\end{align}
\end{lemma}

\begin{lemma}\label{lem:time-reg-integral} Let $\left(E,\norm{\cdot}_E \right)$ be a separable Banach space of type $2$. Then there exists a constant $C >0$ such that for all $q \in \com{[1},\infty)$, $r \in (2,\infty]$\com{, $\theta \in (1,2]$ and $G \in L^{q\theta}_\mathcal{F} L^r_t \gamma(U;E)$ } we have 
\begin{align} \label{eq:Besov-reg}
\left( \mathbb{E}\left[ \norm{\mathcal{I}(G)}_{B_{q,\infty}^\alpha(0,T;E)}^{ q\theta } \right] \right)^\frac{1}{q\theta} \leq C \com{ (\theta-1)^{-\frac{1}{q\theta}} \sqrt{q} } \norm{G}_{L_\omega^{ q \theta} L^r_t \gamma(U;E)},
\end{align}
where $\alpha = \frac{1}{2} - \frac{1}{r}$.
\end{lemma}

The subsequent proof is motivated by the proof of Theorem 3.2 in \cite{MR4116708}.
\begin{proof}
Without loss of generality we assume $T=1$, the general case follows by a scaling argument. Define 
\begin{align*}
Y_{n,q} := 2^{n\alpha} \norm{\mathcal{I}(G)(\cdot + 2^{-n}) - \mathcal{I}(G)(\cdot)}_{L^q(I \cap I-2^{-n}); E)}.
\end{align*}
Note, we may rewrite
\begin{align*}
Y_{n,q}^q &= \int_0^{1-2^{-n}} 2^{n\alpha q} \norm{\mathcal{I}(G)(t + 2^{-n}) -\mathcal{I}(G)(t)}_{ E}^q \dt  \\
&= \sum_{m=1}^{2^n-1} \int_{(m-1)2^{-n}}^{m2^{-n}} 2^{n\alpha q} \norm{\mathcal{I}(G)(t + 2^{-n}) -\mathcal{I}(G)(t)}_{ E}^q  \dt \\
&=  \sum_{m=1}^{2^n-1} 2^{-n} \int_{0}^{1} 2^{n\alpha q} \norm{\mathcal{I}(G)((s+m)2^{-n}) -\mathcal{I}(G)((s+m-1)2^{-n})}_{ E}^q  \ds\\
&=: \int_{0}^{1} 2^{-n}\sum_{m=1}^{2^n-1}  \eta_{n,m,s}  \ds.
\end{align*}
Furthermore define
\begin{align*}
Z_{n,q}^q := \int_{0}^{1} 2^{-n}\sum_{m=1}^{2^n-1}  \zeta_{n,m,s}  \ds
\end{align*}
with $\zeta_{n,m,s} :=  \mathbb{E} \left[\eta_{n,m,s} | \mathcal{F}_{(s+m-1)2^{-n}} \right] $.
For \com{a } fixed $n \in \mathbb{N}$ and $s \in [0,1]$ we define for $M \in \set{1,\ldots,2^{n}-1}$
\begin{align*}
\mathcal{E}_{n,M,s}:= \sum_{m=1}^M \eta_{n,m,s} - \zeta_{n,m,s}.
\end{align*}
Observe $\mathcal{E}_{n,\cdot,s}$ is a discrete martingale with respect to the filtration $\mathcal{F}_{(s + \cdot)2^{-n}}$\com{. Indeed, for $k \in \set{1,\ldots, 2^n-1}$, we have
\begin{align*}
&\mathbb{E} \left[\mathcal{E}_{n,M,s} \big| \mathcal{F}_{(s+k)2^{-n}} \right] \\
&= \mathbb{E} \left[\sum_{m=1}^{k} \eta_{n,m,s} - \zeta_{n,m,s} \Bigg| \mathcal{F}_{(s+k)2^{-n}} \right] + \mathbb{E} \left[\sum_{m=k+1}^{M} \eta_{n,m,s} - \zeta_{n,m,s} \Bigg| \mathcal{F}_{(s+k)2^{-n}} \right] \\
&=: \mathrm{I} + \mathrm{II}. 
\end{align*}
Recall $\mathcal{I}(G)$ is adapted to $(\mathcal{F}_t)$. Therefore $\eta_{n,m,s}$ is $\mathcal{F}_{(s+k)2^{-n}}$-measurable for all $m \leq k$. Now, due to measurability, we conclude
\begin{align*}
\mathrm{I} = \mathcal{E}_{n,k,s}.
\end{align*}
The tower property of conditional expectation implies, for all $m \geq k+1$,
\begin{align*}
 \mathbb{E} \left[ \zeta_{n,m,s} \big| \mathcal{F}_{(s+k)2^{-n}} \right] &=  \mathbb{E} \left[ \mathbb{E} \left[\eta_{n,m,s} | \mathcal{F}_{(s+m-1)2^{-n}} \right]  \big| \mathcal{F}_{(s+k)2^{-n}} \right] \\
 &=   \mathbb{E} \left[\eta_{n,m,s} \big| \mathcal{F}_{(s+k)2^{-n}} \right].
\end{align*}
Thus $\mathrm{II} = 0.$
}

 Fix $\theta \in (1,2]$. We obtain, due to Jensen's inequality and Fubini's Theorem,
\begin{align*}
\mathbb{E} \left[ \abs{Y_{n,q}^q - Z_{n,q}^q}^\theta \right]& = \mathbb{E} \left[ \abs{\int_0^1 2^{-n} \sum_{m=1}^{2^n-1} \eta_{n,m,s} - \zeta_{n,m,s} \ds}^\theta \right]\\
&\leq  \int_0^1 2^{-\theta n} \mathbb{E} \left[ \abs{ \sum_{m=1}^{2^n-1} \eta_{n,m,s} - \zeta_{n,m,s}}^\theta \right]\ds \\
&=  \int_0^1 2^{-\theta n} \mathbb{E} \left[ \abs{ \mathcal{E}_{n,2^{n}-1,s}}^\theta \right]\ds \\
&\leq  \int_0^1 2^{-\theta n} \mathbb{E} \left[\max_{M\in \set{1,\ldots,2^n-1}} \abs{ \mathcal{E}_{n,M,s}}^\theta \right]\ds.
\end{align*}
Applying the classical BDG-inequality and $\ell^{\theta} \hookrightarrow \ell^2$ we get
\begin{align*}
\mathbb{E} \left[\max_{M\in \set{1,\ldots,2^n-1}} \abs{ \mathcal{E}_{n,M,s}}^\theta \right] &\lesssim \mathbb{E} \left[\left(\sum_{m=1}^{2^n -1} (\eta_{n,m,s} - \zeta_{n,m,s})^2 \right)^\frac{\theta}{2}\right] \\
&\leq  \mathbb{E} \left[\sum_{m=1}^{2^n-1} \abs{\eta_{n,m,s} - \zeta_{n,m,s}}^{\theta} \right] .
\end{align*}
Jensen's inequality implies
\begin{align*}
 \mathbb{E} \left[\sum_{m=1}^{2^n-1} \abs{\eta_{n,m,s} - \zeta_{n,m,s}}^{\theta}\right] \leq 2^\theta \sum_{m=1}^{2^n-1} \mathbb{E} \left[ \abs{\eta_{n,m,s}}^{\theta}\right].
\end{align*}
It remains to invoke the BDG-inequality in type $2$ Banach spaces and H\"older's inequality to conclude

\com{
\begin{align*}
\mathbb{E} \left[ \abs{\eta_{n,m,s}}^\theta \right] &\lesssim (\theta q)^\frac{\theta q}{2} \mathbb{E} \left[ 2^{n\alpha q \theta} \left( \int_{(s+m-1)2^{-n}}^{(s+m)2^{-n}} \norm{G}_{\gamma(U;E)}^2 \dd s \right)^\frac{q \theta}{2} \right] \\
&\leq (\theta q)^\frac{\theta q}{2} \mathbb{E}  \left[ 2^{n\alpha q \theta} \left( \int_{(s+m-1)2^{-n}}^{(s+m)2^{-n}} \norm{G}_{\gamma(U;E)}^r \dd s \right)^\frac{q \theta}{r} \left(2^{-n}  \right)^\frac{q\theta(r-2)}{2r} \right]  \\
&= (\theta q)^\frac{\theta q}{2} 2^{n q \theta(\alpha - (1/2 - 1/r))} \mathbb{E}  \left[  \left( \int_{(s+m-1)2^{-n}}^{(s+m)2^{-n}} \norm{G}_{\gamma(U;E)}^r \dd s \right)^\frac{q \theta}{r}  \right] 
\end{align*}
Recall that $\alpha = 1/2 - 1/r$. Collecting the previous estimates, we obtain
\begin{align*}
\sum_{m=1}^{2^n-1} \mathbb{E} \left[ \abs{\eta_{n,m,s}}^\theta \right]  \lesssim (\theta q)^\frac{\theta q}{2} 2^n \mathbb{E}  \left[  \left( \int_{0}^{1} \norm{G}_{\gamma(U;E)}^r \dd s \right)^\frac{q \theta}{r}  \right].
\end{align*}
Now, estimating the geometric sum 
\begin{align*}
\mathbb{E} \left[ \sup_{n \in \mathbb{N}} \abs{Y_{n,q}^q - Z_{n,q}^q}^\theta \right] &\leq \mathbb{E} \left[ \sum_{n \in \mathbb{N}} \abs{Y_{n,q}^q - Z_{n,q}^q}^\theta \right] \\
&\lesssim 2^{\theta} (\theta q)^\frac{\theta q}{2}  \mathbb{E}  \left[  \left( \int_{0}^{1} \norm{G}_{\gamma(U;E)}^r \dd s \right)^\frac{q \theta}{r}  \right] \sum_{n \in \mathbb{N}} 2^{-n(\theta -1)} \\
&\leq 2^{\theta} (\theta q)^\frac{\theta q}{2}  \mathbb{E}  \left[  \left( \int_{0}^{1} \norm{G}_{\gamma(U;E)}^r \dd s \right)^\frac{q \theta}{r}  \right] \frac{1}{1 - 2^{-(\theta - 1)}}.
\end{align*}
Note, for $\theta \geq 1$,
\begin{align*}
1-2^{-(\theta-1)} \geq 2^{-(\theta-1)}\ln(2) (\theta - 1).
\end{align*}
Thus,
\begin{align}
\mathbb{E} \left[ \sup_{n \in \mathbb{N}} \abs{Y_{n,q}^q - Z_{n,q}^q}^\theta \right] \lesssim \frac{2^{2\theta-1} (\theta q)^\frac{\theta q}{2}  }{\ln(2)(\theta -1)}\mathbb{E}  \left[  \left( \int_{0}^{1} \norm{G}_{\gamma(U;E)}^r \dd s \right)^\frac{q \theta}{r}  \right].
\end{align}

Lastly we estimate $Z_{n,q}$. Due to Lemma \ref{lem:reg-cond-exp} it holds $\mathbb{P}$-a.s.
\begin{align*}
Z_{n,q}^q \leq \int_0^1 \max_{1\leq m \leq 2^n-1} \zeta_{n,m,s} \ds \lesssim  q^\frac{q}{2} 2^{nq(\alpha - (1/2 - 1/r))} \norm{G}_{L^q_\omega L^r_t \gamma(U;E)}^q .
\end{align*}
Take the $\theta$-th power, supremum in $n$ and expectation
\begin{align*}
\mathbb{E} \left[ \sup_{n \in \mathbb{N}} \abs{Z_{n,q}^q}^\theta \right] \lesssim q^\frac{q\theta}{2}\norm{G}_{L^q_\omega L^r_t \gamma(U;E)}^{ q\theta } \leq q^\frac{q\theta}{2}\norm{G}_{L^{q\theta}_\omega L^r_t \gamma(U;E)}^{ q\theta } .
\end{align*}
The semi-norm esimate now follows, since
\begin{align*}
\mathbb{E} \left[ \seminorm{\mathcal{I}(G)}_{B_{q,\infty}^{\alpha}(0,1;E)}^{q\theta} \right] &= \mathbb{E} \left[ \sup_{n \in \mathbb{N}} \abs{Y_{n,q}^q}^\theta \right] \\
&\leq 2^\theta \left( \mathbb{E} \left[ \sup_{n \in \mathbb{N}} \abs{Y_{n,q}^q - Z_{n,q}^q}^\theta \right] + \mathbb{E} \left[ \sup_{n \in \mathbb{N}} \abs{Z_{n,q}^q}^\theta \right] \right) \\
&\lesssim 2^\theta \left( \frac{2^{2\theta-1} \theta^\frac{\theta q}{2}  }{\ln(2)(\theta -1)} + 1 \right) q^\frac{q\theta}{2}\norm{G}_{L^{ q\theta}_\omega L^r_t \gamma(U;E)}^{ q \theta}.
\end{align*}
Applying the BDG-inequality in type $2$ Banach spaces and H\"older's inequality 
\begin{align*}
\mathbb{E} \left[ \norm{\mathcal{I}(G)}_{L^q(0,1;E)}^{q\theta} \right] &\lesssim (q \theta)^\frac{q \theta}{2} \mathbb{E}\left[ \left( \int_0^1 \norm{G}_{\gamma(U;E)}^2 \ds \right)^\frac{q\theta}{2} \right] \\
&\leq (q \theta)^\frac{q \theta}{2} \norm{G}_{L^{q\theta}_\omega L^r_t \gamma(U;E)}^{q\theta}.
\end{align*}
Overall we have proved
\begin{align*}
\mathbb{E} \left[ \norm{\mathcal{I}(G)}_{B^\alpha_{q,\infty}(0,1;E)}^{ q \theta} \right] &\leq 2^{\theta} \left(  \mathbb{E} \left[ \norm{\mathcal{I}(G)}_{L^{ q}(0,1;E)}^{q\theta} \right] + \mathbb{E} \left[ \seminorm{\mathcal{I}(G)}_{B_{q,\infty}^{\alpha}(0,1;E)}^{q\theta} \right] \right) \\
&\lesssim 2^\theta \left( \frac{2^{2\theta-1} \theta^\frac{\theta q}{2}  }{\ln(2)(\theta -1)} +\theta^\frac{\theta q}{2}  +1 \right)q^\frac{q\theta}{2}\norm{G}_{L^{ q\theta}_\omega L^r_t \gamma(U;E)}^{ q\theta}\\
&\leq \frac{1}{\theta-1} 2^\theta\left( \frac{2^{2\theta-1} \theta^\frac{\theta q}{2}  }{\ln(2)} +\theta^\frac{\theta q}{2}  +1 \right)q^\frac{q\theta}{2}\norm{G}_{L^{ q\theta}_\omega L^r_t \gamma(U;E)}^{ q\theta}.
\end{align*}
The assertion follows by taking the $q\theta$-th root.
}

\end{proof}
\begin{remark}
The statements for the time regularity in the critical regime $\alpha = 1/2 - 1/r$ are delicate. For $\alpha < 1/2-1/r$ one easily proves
\begin{align*}
\left( \mathbb{E}\left[ \norm{\mathcal{I}(G)}_{B_{q,\infty}^\alpha(0,T;E)}^{q} \right] \right)^\frac{1}{q} \lesssim \norm{G}_{L^{q}_\omega L^r_t \gamma(U;E)},
\end{align*}
using $\ell^1 \hookrightarrow \ell^\infty$ and interchanging integration and summation.
In \cite{MR4116708} Theorem~3.2(ii) the authors prove a slightly weaker estimate
\begin{align*}
\left( \mathbb{E}\left[ \norm{\mathcal{I}(G)}_{B_{q,\infty}^\alpha(0,T;E)}^{2q} \right] \right)^\frac{1}{2q} \leq C_T \sqrt{q} \norm{G}_{L^{2q}_\omega L^r_t \gamma(U;E)}.
\end{align*}
The result can be recovered by our Lemma \ref{lem:time-reg-integral} by setting $\theta = 2$ in the proof. We do not know whether the limit case $\theta = 1$ can be proved by a direct estimate on 
\begin{align*}
\mathbb{E} \left[ \sup_{n \in \mathbb{N}} \abs{Y_{n,p}^p - Z_{n,p}^p} \right].
\end{align*}
\end{remark}

\begin{corollary} \label{cor:stoch-int}
Let $p \geq 2$ and Assumption \ref{ass:Noise} be satisfied. Then for all $q \in \com{[1, }\infty)$ \com{and $\theta \in (1,2]$ }
\begin{align} \label{eq:cor-stoch}
\left( \mathbb{E}\left[ \norm{\mathcal{I}(G(v))}_{B_{q,\infty}^{1/2}(0,T;W^{1,p}_x)}^{q\com{\theta}} \right] \right)^\frac{1}{q\com{\theta}} \lesssim  \norm{v}_{L_\omega^{q\com{\theta}} L^\infty_t W^{1,p}_x} + 1.
\end{align}
\end{corollary}
\begin{proof}
We aim at applying Lemma \ref{lem:time-reg-integral} with $E = W^{1,p}_x$ and $r = \infty$. Similar to Proposition \ref{prop:type2} one checks that $W^{1,p}_x$ is a type $2$ Banach space. Now due to Lemma \ref{lem:time-reg-integral} and Lemma \ref{lem:G-gamma}
\begin{align*}
\left( \mathbb{E}\left[ \norm{\mathcal{I}(G(v))}_{B_{q,\infty}^{1/2}(0,T;W^{1,p}_x)}^{q\com{\theta}} \right] \right)^\frac{1}{q\com{\theta}} &\lesssim \norm{G(v)}_{L_\omega^{q\com{\theta}} L^\infty_t \gamma(U;W^{1,p}_x)} \\
&\lesssim \norm{v}_{L_\omega^{q\com{\theta}} L^\infty_t W^{1,p}_x} + 1.
\end{align*}
\end{proof}
Lemma~\ref{lem:time-reg-integral} and Corollary~\ref{cor:stoch-int} guarantee the stability of the stochastic integral operator $\mathcal{I}$ respectively the compositional operator $\mathcal{I} \circ G$ under a prescribed integrability in probability $L^q_\omega$. We observe that we lose a tiny bit of integrability in time, when prescribing the same integrability in probability of the input and the output space. 
However if we allow for different integrability properties in probability, one can even establish stability of the stochastic integral operator on exponentially integrable Besov spaces in time as done by Ondreját and Veraar in \cite{MR4116708}.
\begin{lemma}[\cite{MR4116708} Theorem 3.2 (vi)]\label{lem:time-reg-integral-2}
Let $\left(E,\norm{\cdot}_E \right)$ be a separable Banach space of type $2$ and $G \in L^q_\mathcal{F} L^r_t \gamma(U;E)$. Then there exists a constant $C >0$ such that for all $q \in \com{[1, }\infty)$ and $r \in (2,\infty]$ we have 
\begin{align} \label{eq:Besov-reg-2}
\left( \mathbb{E} \left[ \norm{\mathcal{I}(G)}_{B^{\alpha}_{\Phi_2,\infty}(0,T;E)}^q \right] \right)^\frac{1}{q} \leq C \sqrt{q} \norm{G}_{L^{N_q}(\Omega;L^r(0,T;\gamma(U;E)))},
\end{align}
where $\alpha = \frac{1}{2} - \frac{1}{r}$, $\Phi_2(t)= e^{t^2}-1$ and $N_q(t) = t^q \log^{q/2}(t+1)$.
\end{lemma}
An immediate consequence is the stability of the compositional operator.
\begin{corollary} \label{cor:stoch-int-2}
Let $p \geq 2$ and Assumption \ref{ass:Noise} be satisfied. Then for all $q \in \com{[1, }\infty)$
\begin{align} \label{eq:cor-stoch-2}
 \norm{\mathcal{I}(G(v))}_{L^{q}_\omega B_{\Phi_2,\infty}^{1/2}(0,T;W^{1,p}_x)} \lesssim  \norm{v}_{L_\omega^{\com{N_q}} L^\infty_t W^{1,p}_x} + 1.
\end{align}
\end{corollary}

Before we state the main result, let us shortly explain where the exponentially integrable spaces are located in terms of the usual scale of Lebesgue spaces. For a general overview we refer to the book \cite{DiHaHaRu}.
\begin{lemma}[\com{\cite{DiHaHaRu} Corollary~3.3.4}] \label{lem:exp-integrable}
Let $\Phi_2(t) = e^{t^2} - 1$. Then 
\begin{align*}
L^\infty(I) \hookrightarrow L^{\Phi_2}(I) \hookrightarrow L^{\infty-}(I)
\end{align*}
with continuous embeddings\com{, i.e. for all $q \in[1,\infty)$ there exist constants $C, C_q >0$ such that for all $u \in L^\infty(I)$ 
\begin{align}
\norm{u}_{L^q(I)} \leq C_q \norm{u}_{L^{\Phi_2}(I)} \leq C \norm{u}_{L^\infty(I)}.
\end{align}
}
\end{lemma}
%
%
%

\subsection{Exponential Besov regularity} \label{sec:Besov_regularity}
We are ready to formulate the first main result on time regularity of strong solutions:

\begin{theorem}[Exponential Besov regularity]\label{thm:Besov-reg}
Let the assumptions of Theorem \ref{thm:strong-sol} be satisfied. Let $u$ be the unique strong solution to \eqref{eq:p-Laplace}. Additionally assume 
\begin{align} \label{ass:thm-Besov}
u \in L^q_\omega L^\infty_t L^2_x
\end{align}
for some $q > 2$. Then $u \in L^2_\omega B^{1/2}_{\Phi_2,\infty} L^2_x$ with
\begin{align}\label{eq:thm-Besov}
\norm{u}_{L^2_\omega B^{1/2}_{\Phi_2,\infty} L^2_x} \lesssim \norm{u_0}_{L^p_\omega W^{1,p}_x}^{p/2} + \norm{u}_{L^q_\omega L^\infty_t L^2_x} + 1,
\end{align}
where $\Phi_2(t) = e^{t^2}-1$.
\end{theorem}

\begin{proof}
Since we assume that $u$ is a strong solution to \eqref{eq:p-Laplace} it holds for almost all $(\omega,t,x)$ 
\begin{align}\label{eq:rep-strong-sol}
u(t) = u_0 + \int_0^t \Div S(\nabla u) \ds + \mathcal{I}(G(u))(t).
\end{align} 
Therefore, due to the triangle inequality and Lemma \ref{lem:exp-integrable}
\begin{align*}
\norm{u}_{L^2_\omega L^{\Phi_2}_t L^2_x} &\leq \norm{u_0}_{L^2_\omega L^{\Phi_2}_t L^2_x} + \norm{\int_0^\cdot \Div S(\nabla u) \ds}_{L^2_\omega L^{\Phi_2}_t L^2_x} + \norm{\mathcal{I}(G(u))}_{L^2_\omega L^{\Phi_2}_t L^2_x}  \\
&\lesssim \norm{u_0}_{L^2_\omega L^2_x} + \norm{\Div S(\nabla u) }_{L^2_\omega L^2_t L^2_x} + \norm{\mathcal{I}(G(u))}_{L^2_\omega L^{\Phi_2}_t L^2_x}.
\end{align*}
The first term is bounded by assumption. The second term is bounded by the a priori estimate \eqref{eq:est-strong}. The last term is controlled by Corollary \ref{cor:stoch-int-2}. Thus,
\begin{align*}
\norm{u}_{L^2_\omega L^{\Phi_2}_t L^2_x} \lesssim \norm{u_0}_{L^2_\omega L^2_x} + \norm{u_0}_{L^p_\omega W^{1,p}_x}^{p/2} + 1 + \norm{u}_{L^q_\omega L^\infty_t L^2_x}.
\end{align*}

It remains to estimate the Besov-Orlicz semi-norm
\begin{align*}
\mathbb{E} \left[ \seminorm{u}_{B^{1/2}_{\Phi_2,\infty}(I;L^2_x)}^2 \right] = \mathbb{E} \left[ \left( \sup_h h^{-1/2} \norm{\tau_h u}_{L^{\Phi_2}(I \cap I-\set{h}; L^2_x)} \right)^2 \right].
\end{align*}
Using the representation \eqref{eq:rep-strong-sol}
\begin{align*}
\mathbb{E} \left[ \seminorm{u}_{B^{1/2}_{\Phi_2,\infty}(I;L^2_x)}^2 \right] &\lesssim \mathbb{E} \left[ \left( \sup_h h^{-1/2} \norm{\int_{\cdot}^{\cdot + h} \Div S(\nabla u) \ds  }_{L^{\Phi_2}(I \cap I-\set{h}; L^2_x)} \right)^2 \right] \\
&\quad + \mathbb{E} \left[ \left( \sup_h h^{-1/2} \norm{\tau_h \mathcal{I}(G(u))}_{L^{\Phi_2}(I \cap I-\set{h}; L^2_x)}  \right)^2 \right] \\
&=: \mathrm{I} + \mathrm{II}.
\end{align*}
The first term can be estimated due to the $L^2_\omega L^2_t$ integrability of the expression $\norm{\Div S(\nabla u)}_{L^2_x}$. Invoking Lemma \ref{lem:exp-integrable} and H\"older's inequality
\begin{align*}
\mathrm{I} &\lesssim \mathbb{E} \left[ \left( \sup_h h^{-1/2} \sup_{t \in I \cap I-\set{h}} \int_{t}^{t + h} \norm{\Div S(\nabla u)  }_{L^2_x} \ds \right)^2 \right] \\
&\leq \mathbb{E} \left[  \sup_h \sup_{t \in I \cap I-\set{h}} \int_{t}^{t + h} \norm{\Div S(\nabla u)  }_{L^2_x}^2 \ds  \right] \\
&\leq \norm{\Div S(\nabla u)}_{L^2_\omega L^2_t L^2_x}^2 \\
&\lesssim \norm{u_0}_{L^p_\omega W^{1,p}_x}^{p} + 1,
\end{align*}
where in the last line we applied the a priori estimate \eqref{eq:est-strong}. The second term is bounded via Corollary \ref{cor:stoch-int-2}
\begin{align*}
\mathrm{II} &\leq \norm{\mathcal{I}(G(u))}_{L^{2}_\omega B_{\Phi_2,\infty}^{1/2}(0,T;L^2_x)}^2  \lesssim \norm{u}_{L^q_\omega L^\infty_t L^2_x}^2 + 1.
\end{align*}
Overall, the claim \eqref{eq:thm-Besov} follows and the proof is finished.
\end{proof}

\begin{corollary} \label{cor:cor-hoelder}
Let the assumptions of Theorem \ref{thm:Besov-reg} be satisfied. Then it holds $u \in L^2_\omega C^{1/2-}_t L^2_x$ and
\begin{align}\label{eq:cor-hoelder}
\norm{u}_{L^2_\omega C^{1/2-}_t L^2_x} \lesssim \norm{u_0}_{L^p_\omega W^{1,p}_x}^{p/2} + \norm{u}_{L^q_\omega L^\infty_t L^2_x} + 1.
\end{align}
\end{corollary}

\begin{remark}
Hyt\"onen and Veraar have obtained a precise statement about time regularity in terms of Besov spaces for Banach space valued Brownian motions in \cite{VerHyt08}. Indeed, they show that any $E$-valued Brownian motion $W$ satisfies
\begin{align*}
W \in B^{1/2}_{\Phi_2,\infty}(I;E) \quad \text{ a.s. }
\end{align*}
Theorem \ref{thm:Besov-reg} proves that the $L^2_x$-valued solution process $u$ is as regular in time as a $L^2_x$-valued Brownian motion. 

Additionally, the divergence of the nonlinear operator in the representation \eqref{eq:rep-strong-sol} has only slightly higher time regularity in comparision to the stochastic term, i.e. 
\begin{align*}
(\omega,t,x) \mapsto \int_0^t \Div S(\nabla u) \ds \in L^2_\omega W^{1,2}_t L^2_x \hookrightarrow L^2_\omega C^{1/2}_t L^2_x \hookrightarrow L^2_\omega B^{1/2}_{\Phi_2,\infty} L^2_x,
\end{align*}
with continuous embeddings. However, as Corollary \ref{cor:stoch-int-2} suggests, the stochastic integral even has Besov regularity on the gradient level. 

The assumption \eqref{ass:thm-Besov} is not restrictive. It can be verified under suitable integrability assumptions on the initial condition, e.g. $u_0 \in L^q_\omega L^2_x$ is sufficient to ensure 
\begin{align*}
\mathbb{E} \left[\sup_t \norm{u}_{L^2_x}^q \right] \leq c_q \left( \mathbb{E} \left[ \norm{u_0}_{ L^2_x}^q \right] + 1 \right).
\end{align*}
\end{remark}

\subsection{Nikolskii regularity of nonlinear gradient} \label{sec:Nikolskii_regularity}
The key ingredient while deriving estimates for the nonlinear term $V(\nabla u) $ is the $V$-coercivity Lemma \ref{lem:V-coercive} that relates the nonlinear operators $S$ and $V$. Now we present the second main result:


\begin{theorem}[Nikolskii regularity for nonlinear gradient] \label{thm:time-reg-sol}
Let the assumptions of Theorem \ref{thm:strong-sol} be satisfied. Let $u$ be the unique strong solution to \eqref{eq:p-Laplace}. Additionally assume 
\begin{align} \label{ass:thm-Nikolskii}
u \in L^q_\omega L^\infty_t W^{1,p}_x
\end{align}
for some $q > p$. Then $V(\nabla u) \in L^2_\omega B^{1/2}_{2,\infty} L^2_x$ with 
\begin{align}\label{eq:thm-Nikolskii}
\norm{V(\nabla u)}_{L^2_\omega B^{1/2}_{2,\infty} L^2_x}^2 \lesssim \norm{u_0}_{L^p_\omega W^{1,p}_x}^p + \norm{u}_{L^q_\omega L^\infty_t W^{1,p}_x}^p +1.
\end{align}
\end{theorem}

\begin{proof}
Let $u$ be the strong solution to \eqref{eq:p-Laplace}. Take $t,h \in I$ such that $t+h \in I$. Due to Lemma \ref{lem:V-coercive} it holds
\begin{align} \label{eq:thV}
\norm{\tau_h \left( V(\nabla u)\right)(t)}_{L^2_x}^2 &\eqsim \int_{\mathcal{O}} \tau_h \left( S(\nabla u) \right)(t) : \tau_h (\nabla u)(t) \dx.
\end{align}
Integration by parts and the solution formula \eqref{eq:rep-strong-sol} then imply
\begin{align} \label{eq:I+II}
\begin{aligned}
&\norm{\tau_h \left( V(\nabla u)\right)(t)}_{L^2_x}^2
\eqsim - \int_{\mathcal{O}} \tau_h \left(  \Div S(\nabla u) \right)(t) : \tau_h (u)(t) \dx \\
&\quad = - \int_{\mathcal{O}} \tau_h \left(  \Div S(\nabla u) \right)(t) : \left( \int_{t}^{t+h} \Div S(\nabla u) \ds + \tau_h (\mathcal{I}(G(u)))(t) \right) \dx \\
&\quad =: \mathrm{I} + \mathrm{II}. 
\end{aligned}
\end{align}
The first \com{term } can be estimated by H\"older's and Young's inequalities
\begin{align} \label{eq:est-1}
\begin{aligned}
\mathrm{I} &\leq \norm{\tau_h (\Div S(\nabla u))(t)}_{L^2_x} \norm{\int_{t}^{t+h} \Div S(\nabla u) \ds }_{L^2_x} \\
&\leq h \left( \norm{\tau_h (\Div S(\nabla u))(t)}_{L^2_x}^2 +\dashint_{t}^{t+h} \norm{\Div S(\nabla u) }_{L^2_x}^2 \ds  \right).
\end{aligned}
\end{align}
The second term needs a more sophisticated analysis. Due to the \com{A}ssumption~\ref{ass:Noise}~\ref{ass:boundary-noise} the stochastic integral preserves zero boundary values. Therefore, integration by parts and Lemma~\ref{lem:1side-young} reveal
\begin{align} \label{eq:II}
\begin{aligned}
\mathrm{II} &= \int_{\mathcal{O}} \tau_h (S(\nabla u)) (t) : \tau_h ( \nabla \mathcal{I}(G(u)))(t) \dx \\
&\leq \delta \norm{\tau_h V(\nabla u)(t)}_{L^2_x}^2 \\
&\quad + c_\delta \int_{\mathcal{O}} \left( \kappa + \abs{\nabla u(t)} + \abs{\tau_h (\nabla u)(t)} \right)^{p-2} \abs{\tau_h( \nabla \mathcal{I}(G(u)))(t)}^2 \dx\\
&=: \mathrm{II}_1 + \mathrm{II}_2.
\end{aligned}
\end{align}
Since $\delta >0$ is arbitary, we can absorb the first term \com{ of \eqref{eq:II} } to the left hand side.

Let us first discuss the case $p > 2$. Here the latter one is estimated by the weighted Young's inequality with exponents $\theta = p/(p-2)$ and $\theta' = p/2$ to get
\begin{align} \label{eq:est-2}
\mathrm{II}_2 &\lesssim h \left( \frac{p-2}{p} \left( \sup_t \norm{\nabla u}_{L^p_x}^p + 1\right) + \frac{2}{p} h^{-p/2} \norm{\tau_h \mathcal{I}(G(u))(t)}_{W^{1,p}_x}^p \right).
\end{align}
In the case $p=2$ it trivially holds 
\begin{align} \label{eq:est-2b}
\mathrm{II}_2 &\lesssim \norm{\tau_h \mathcal{I}(G(u))(t)}_{W^{1,2}_x}^2,
\end{align}
which coincides with \eqref{eq:est-2} for $p=2$.

Overall, we integrate \eqref{eq:thV} over $I \cap I- \set{h}$ and multiply by $h^{-1}$. Applying the estimates \eqref{eq:est-1} and \eqref{eq:est-2}, we arrive at
\begin{align*}
h^{-1} &\int_{I \cap I- \set{h}} \norm{\tau_h V(\nabla u)(t)}_{L^2_x}^2 \dt \\
&\lesssim \int_{I \cap I- \set{h}} \left( \norm{\tau_h (\Div S(\nabla u))(t)}_{L^2_x}^2 +\dashint_{t}^{t+h} \norm{\Div S(\nabla u) }_{L^2_x}^2 \ds \right) \dt \\
&+ \int_{I \cap I- \set{h}}\left( \frac{p-2}{p} \left( \sup_t \norm{\nabla u}_{L^p_x}^p + 1\right) + \frac{2}{p} h^{-p/2} \norm{\tau_h \mathcal{I}(G(u))(t)}_{W^{1,p}_x}^p \right) \dt.
\end{align*}
Finally, take the supremum in $h$ and expectation. Due to Fubini's theorem 
\begin{align*}
&\mathbb{E} \left[ \seminorm{V(\nabla u)}_{B^{1/2}_{2,\infty} L^2_x}^2 \right] = \mathbb{E} \left[ \left( \sup_h h^{-1/2} \left( \int_{I \cap I- \set{h}} \norm{\tau_h V(\nabla u)(t)}_{L^2_x}^2 \dt \right)^\frac{1}{2} \right)^2 \right] \\
&\quad \lesssim \norm{ \Div S(\nabla u)}_{L^2_\omega L^2_t L^2_x}^2 + \frac{p-2}{p} \left( \norm{\nabla u}_{L^p_\omega L^\infty_t L^p_x}^p +1 \right) \\
&\quad + \frac{2}{p}\mathbb{E} \left[ \left( \sup_h h^{-1/2} \left( \int_{I \cap I-\set{h}} \norm{\tau_h \mathcal{I}(G(u))}_{W^{1,p}_x}^p \ds \right)^\frac{1}{p} \right)^p \right]\\
&\leq \norm{ \Div S(\nabla u)}_{L^2_\omega L^2_t L^2_x}^2 + \frac{p-2}{p} \left( \norm{\nabla u}_{L^p_\omega L^\infty_t L^p_x}^p +1 \right) +\frac{2}{p} \norm{\mathcal{I}(G(u))}_{L^p_\omega B_{p,\infty}^{1/2} W^{1,p}_x}^p.
\end{align*}
Lastly, it remains to use Lemma \ref{lem:exp-integrable} and Corollary \ref{cor:stoch-int-2} with $q = p+\varepsilon$ to bound the third term in the estimate
\begin{align} \label{eq:est3}
\norm{\mathcal{I}(G(u))}_{L^p_\omega B_{p,\infty}^{1/2} W^{1,p}_x}^p \lesssim \norm{\mathcal{I}(G(u))}_{L^p_\omega B_{\Phi_2,\infty}^{1/2} W^{1,p}_x}^p \lesssim \norm{u}_{L^{q}_\omega L^\infty_t W^{1,p}_x}^p + 1.
\end{align}
Alltogether, we have proven
\begin{align*}
\mathbb{E} \left[ \seminorm{V(\nabla u)}_{B^{1/2}_{2,\infty} L^2_x}^2 \right] \lesssim  \norm{ \Div S(\nabla u)}_{L^2_\omega L^2_t L^2_x}^2 + \norm{u}_{L^q_\omega L^\infty_t W^{1,p}_x}^p +1.
\end{align*}
The remaining part of the norm is controlled by H\"older's inequality
\begin{align*}
\norm{V(\nabla u)}_{L^2_\omega L^2_t L^2_x}^2 = \norm{\nabla u}_{L^p_\omega L^p_t L^p_x}^p \lesssim \norm{\nabla u}_{L^q_\omega L^\infty_t L^p_x}^p.
\end{align*}
Overall, using the a priori estimate \eqref{eq:est-strong}, we arrive at
\begin{align*}
\norm{V(\nabla u)}_{L^2_\omega B_{2,\infty}^{1/2} L^2_x}^2 &\lesssim \norm{ \Div S(\nabla u)}_{L^2_\omega L^2_t L^2_x}^2 + \norm{u}_{L^q_\omega L^\infty_t W^{1,p}_x}^p +1\\
&\lesssim \norm{u_0}_{L^p_\omega W^{1,p}_x}^p + \norm{u}_{L^q_\omega L^\infty_t W^{1,p}_x}^p +1.
\end{align*}
The assertion is proved.
\end{proof}

\begin{remark}
The only reason, why we assume condition \ref{ass:boundary-noise} in Assumption \ref{ass:Noise}, is, so that we can revert the partial integration after using the strong formulation of the equation \eqref{eq:p-Laplace}.

The additional assumption \eqref{ass:thm-Nikolskii} is also not very restrictive. It can be verified with appropiate assumptions on the initial condition, e.g. $u_0 \in L^q_\omega W^{1,p}_x$ for some $q > p$ and follows in the same line as the proof of Theorem \ref{thm:strong-sol}.

If we do not want to allow for an increased integrability assumption on $u$ as \eqref{ass:thm-Nikolskii}, we still obtain regularity estimates by using Corollary \ref{cor:stoch-int} instead of Corollary \ref{cor:stoch-int-2} in the proof of Theorem \ref{thm:time-reg-sol}
\begin{align*}
\norm{V(\nabla u)}_{L^2_\omega B^{1/2}_{2-,\infty} L^2_x}^2 \lesssim \norm{u_0}_{L^p_\omega W^{1,p}_x}^p + \norm{u}_{L^p_\omega L^\infty_t W^{1,p}_x}^p +1.
\end{align*}

We want to point out, that one can use the regularity estimate \eqref{eq:thm-Besov} together with \eqref{eq:thm-Nikolskii} in the analysis of numerical algorithms. In \cite{MR4298537} the authors use $\alpha$-H\"older regularity of the solution process $u$ and $\alpha$-fractional Sobolev regularity of $V(\nabla u)$ to obtain convergence of order $\tau^{\alpha}$, $\alpha < 1/2$. If one instead uses the exponential Besov regularity of $u$ and the Nikolskii regularity of $V(\nabla u)$ this can be improved to $\sqrt{\tau \ln(1+ \tau^{-1})}$.
\end{remark}

\subsection{Higher order Nikolskii regularity of nonlinear gradient} \label{sec:Higher_order_Nikolskii}
Suprisingly the time regularity of the nonlinear gradient $V(\nabla u)$ is restricted due to lack of time integrability of the nonlinear diffusion $\norm{\Div S(\nabla u)}_{L^2_x} \in L^q(0,T)$ for $q >2$ rather than the reduced time regularity of the stochastic integral.
\begin{theorem}[Higher gradient regularity] \label{thm:higher-reg}
Let the Assumption \ref{ass:Noise} be satisfied. Additionally, assume
\begin{subequations}
\begin{align}
u &\in L^{q_1}_\omega L^\infty_t W^{1,p}_x,   \\ \label{ass:Div}
\Div S(\nabla u) &\in L^2_\omega L^{q_2}_t L^2_x,
\end{align}
\end{subequations}
for some $q_1 >p$ and $q_2 \geq 2$. Let $u$ be a strong solution to \eqref{eq:p-Laplace}. Then $V(\nabla u) \in L^2_\omega B^{1/2}_{q_2,\infty} L^2_x$ with
\begin{align}
&\norm{V(\nabla u)}_{L^2_\omega B^{1/2}_{q_2,\infty} L^2_x}^2 \lesssim \norm{\Div S(\nabla u)}_{L^2_\omega L^{q_2}_t L^2_x}^2 + \norm{ \nabla u}_{L^{q_1}_\omega L^\infty_t L^p_x}^p +1.
\end{align}
\end{theorem}
\begin{proof}
The proof proceeds similar to the proof of Theorem \ref{thm:time-reg-sol}. Let $t,h \in I$ such that $t+h \in I$. Recalling \eqref{eq:I+II}, we arrive at
\begin{align*}
\norm{\tau_h V(\nabla u)}_{L^2_x}^{q_2} \eqsim \left( \mathrm{I} + \mathrm{II} \right)^{q_2/2} \lesssim \mathrm{I}^{q_2/2} + \mathrm{II}^{q_2/2}.
\end{align*}
Due to \eqref{eq:est-1},
\begin{align*}
\mathrm{I}^{q_2/2} \lesssim h^{q_2/2} \left( \norm{\tau_h (\Div S(\nabla u))(t)}_{L^2_x}^{q_2} +\dashint_{t}^{t+h} \norm{\Div S(\nabla u) }_{L^2_x}^{q_2} \ds  \right).
\end{align*}
The second term is estimated as in \eqref{eq:II} and \eqref{eq:est-2}
\begin{align*}
&\mathrm{II}^{q_2/2} \lesssim \delta \norm{\tau_h V(\nabla u)}_{L^2_x}^{q_2}  \\
&+ h^{q_2/2} \left( \frac{p-2}{p} \left( \sup_t \norm{\nabla u}_{L^p_x}^{pq_2/2} + 1\right) + \frac{2}{p} \left( h^{-1/2} \norm{\tau_h \mathcal{I}(G(u))(t)}_{W^{1,p}_x} \right)^{p q_2/2} \right).
\end{align*}
Choosing $\delta > 0 $ sufficiently small, we may absorb the first term to the left hand side. Multiplication by $h^{-q_2/2}$ yields
\begin{align*}
&h^{-q_2/2} \norm{\tau_h V(\nabla u)}_{L^2_x}^{q_2}  \lesssim \norm{\tau_h (\Div S(\nabla u))(t)}_{L^2_x}^{q_2} +\dashint_{t}^{t+h} \norm{\Div S(\nabla u) }_{L^2_x}^{q_2} \ds \\
&\quad + \sup_t \norm{\nabla u}_{L^p_x}^{pq_2/2} + 1 + \left( h^{-1/2} \norm{\tau_h \mathcal{I}(G(u))(t)}_{W^{1,p}_x} \right)^{p q_2/2}.
\end{align*}
Finally, integrate the parameter $t$ in time, take the $2/q_2$-th power, supremum over $h$ and expectation
\begin{align*}
\mathbb{E} \left[ \seminorm{V(\nabla u)}_{B^{1/2}_{q_2,\infty} L^2_x}^2 \right] &\lesssim \norm{\Div S(\nabla u)}_{L^2_\omega L^{q_2}_t L^2_x}^2 + \norm{\nabla u}_{L^p_\omega L^\infty_t L^p_x}^p \\
&\quad + \mathbb{E} \left[ \seminorm{\mathcal{I}(G(u))}_{B^{1/2}_{pq_2/2,\infty} W^{1,p}_x}^p \right] + 1.
\end{align*}
The seminorm estimate now follows by an application of \eqref{eq:est3}. H\"older's inequality establishes the remainding estimate
\begin{align*}
\norm{V(\nabla u)}_{L^2_\omega L^{q_2}_t L^2_x}^2 = \norm{\nabla u}_{L^p_\omega L^{p q_2/2}_t L^p_x}^p \lesssim \norm{\nabla u}_{L^p_\omega L^{\infty}_t L^p_x}^p.
\end{align*}
The assertion is proved.
\end{proof}

\begin{remark}
It is not know whether the assumption \eqref{ass:Div} can be verified under appropiate assumptions on the domain and the noise coefficient $G$ as done in Theorem \ref{thm:strong-sol} for the case $q_2 = 2$. 

An immediate consequence is the H\"older regularity of the map $t \mapsto V(\nabla u)(t)$ as an $L^2_x$-valued process, since
\begin{align*}
\norm{V(\nabla u)}_{L^2_\omega C^{1/2 - 1/q_2}_t L^2_x}^2 \lesssim \norm{V(\nabla u)}_{L^2_\omega B^{1/2}_{q_2,\infty} L^2_x}^2.
\end{align*}
\end{remark}

\bibliographystyle{amsalpha}
\bibliography{numerics}

\end{document}